\documentclass[12pt,a4paper]{article} 

\usepackage{amsmath,latexsym,amssymb}

\usepackage{euscript}
\usepackage{eufrak}

\usepackage{mathrsfs}
\usepackage{multicol}

\newenvironment{proof}{\noindent\bf Proof. \rm}{\hfill $\mbox{\boldmath{$ \square$}}$}

\newtheorem{coro}{\bf Corollary}[section]
\newtheorem{theo}[coro]{\bf Theorem}
\newtheorem{defi}[coro]{\bf Definition}
\newtheorem{lem}[coro]{\bf Lemma}
\newtheorem{lemma}[coro]{\bf Lemma}

\def\S{\mathbf{S}}
\def\P{\mathcal{P}}
\def\F{\mathcal{F}}

\title{ \large \textbf{ Leibniz's law and its paraconsistent models}}

\author{Aldo Figallo-Orellano  }
\date{\small Centre for Logic, Epistemology and the History of
Science (CLE),\\  University of Campinas (Unicamp),Brazil\\
{\small E-mail:~{aldofigallo@gmail.com}}\\[1mm]}

\begin{document}
\maketitle
\begin{abstract}
This paper  aims at discussing the importance of Leibniz Law to getting models for Paraconsistent Set Theories.
 
\end{abstract}
\tableofcontents

\section{Introduction}

Paraconsistency is the study of logic systems having a negation $\neg$ which is not explosive; that is, there exist formulas $\alpha$ and $\beta$ in the language of the logic such that $\beta$  is not derivable from the contradictory set $\{\alpha,\neg\alpha\}$. In other words, the logic has contradictory, but non-trivial theories. There are several approaches to paraconsistency in the literature since the introduction of Jaskowski's system of Discursive logic such as Relevant logics, Adaptive logics, Many-valued logics, and many others in 1948. The well-known 3-valued logic  of Paradox (LP) was introduced by  Priest  with the aim of formalizing the philosophical perspective underlying  Priest and  Sylvan's Dialetheism. As it is well-known, the main thesis behind Dialetheism is that there are true contradictions, that is, that some sentences can be both true and false at the same time and in the same way. The logic LP has been intensively studied and developed by several authors proposing, particularly, extensions to first-order languages and applications to Set Theory.

The 1963 publication  of  da Costa's Habilitation thesis “Sistemas Formais Inconsistentes” constitutes a landmark in the history of paraconsistency. In that thesis, da Costa introduced the hierarchy $C_n$ (for $n\leq 1$)  and $C_\omega$  of C-systems, \cite{dC}.

Recall that  $C_\omega$ is defined over  the signature $\Sigma=\{\to,\wedge,\vee,\neg\}$ and  the language  ${\cal L}_\Sigma$  determined by the  Hilbert calculus from axiom schemas from {\em Intuitionistic Positive Calculus}, the rule {\em modus ponens} and the following axiom schemata: ($C\omega$ 1) $\alpha\vee\neg\alpha$ and ($C\omega$ 2) $\neg\neg\alpha\to\alpha$.

In 1969, da Costa visited  {\em Universidad Nacional del Sur} and suggested  finding  a semantics for $C_n$ and $C_\omega$ to Fidel.  In that time, they knew that the negation $\neg$ was not congruencial. In fact, as we proved in \cite{F1},  $C_w$ is not algebrizable with Blok-Pigozzi's method. Fidel overcame this difficulty by means of a presentation of a novel algebraic-relational class of structures called F-structure by adapting  Lindenbaum-Tarski method in order to  prove completeness theorems. The $F$-structure are pairs  $\langle A, \{N_x\}_{x\in A} \rangle$ where $A$ is a generalized Heyting algebra and  $N_x$   is a set of all possible negation of $x\in A$.  The algebraic part  of the structures captures the algebrizable fragment of the system, that is to say, the negation-free fragment. 

In our paper \cite{FOS2}, we apply Fidel's method in order to prove an adequacy theorem, in the strong version  for $C_\omega$ and we also present models for first-order $C_\omega$ (Q$C_\omega$) logic by adapting our work developed in \cite{FOS1}.

The paper is organized as follow: In the section 2, we do a brief review state of art of the classical Set Theory and  non-classical Set Theories in the setting of the existences of their model. In section 3, first we do a 
summary of the known Paraconsistent Set Theories and we discuss the importance of Leibniz Law to obtain models for da Costa's Paraconsistent Set theory. Later on, we analyze the minimal conditions  for a models constructed over a Heyting algebras that we need to prove that several Zermelo-Fraenkel's set-theoretic axioms are valid in a suitable algebraic-like models and finally, we present models for Paraconsistent Nelon's Set Theory, all this is part of the section 4 and 5.
 
\section{ Non-classical Set Theory and their models} \label{section2}
In this section we review the non-classical set theories in the literature. First, recall that Boolean-valued models of set theory were introduced by Scott, Solovay and Vopĕnka in 1965; this theory can be found in Bell (2005), see \cite{Bell}. 

Next, we will make a synthesis of  the construction of Bell's book  for Zermelo-Frankeal classical Set Theory, summarizing the fundamental concepts.

We fix a model of set theory V and an Booelan algebra ${\bf A} = (A, \wedge, \vee, 0, 1, \to)$ and construct a universe of names by transfinite recursion: 

$$ {\mathbf{V}_\xi}^{{\bf A}}=\{x: x\, \textrm{\rm  a function and }\, ran(x)\subseteq A \, \,\textrm{\rm and }\,  dom(x)\subseteq \mathbf{V}_\zeta^{{\bf A}}\, \textrm{\rm for some}\,  \zeta< \xi \} $$

and 

$$ {\mathbf{V}}^{{\bf A}}=\{x: x\in {\mathbf{V}_\xi}^{{\bf A}}\, \textrm{\rm  for some } \xi \} $$

The class ${\mathbf{V}}^{{\bf A}}$  is called the Boolean-valued model over ${\bf A}$. We note that this definition does not depend on the algebraic operations in $A$, but only on the set $A$, so any expansion of $A$ to a richer language will give the same class of names ${\mathbf{V}}^{{\bf A}}$. By ${\cal L}_\in$, we denote the first-order language of set theory using only the propositional connectives $\wedge$, $\vee$, $\bot$, and $\to$. We can now expand this language by adding all of the elements of ${\mathbf{V}}^{{\bf A}}$  as constants; the expanded (class-sized) language will be called ${\cal L}_A$. The (meta-)induction principle for ${\mathbf{V}}^{{\bf A}}$ can be proved by a simple induction on the rank function: for every property $\Phi$ of names, if for all $x\in {\mathbf{V}}^{{\bf A}}$, we have  

$$\forall y \in dom(x)(\Phi(y)\, \mbox{implies}\, \Phi(x)),$$
 
 then all names  $x\in {\mathbf{V}}^{{\bf A}}$ have the property $\Phi$.  We can now define a map $||\cdot||$ assigning to each negation formula in ${\cal L}_A$  a truth value in $A$ as follows.

\begin{defi}\label{str}

 For a given complete  Boolean algebra $\bf A$.  If $u, v\in {\mathbf{V}}^{{\bf A}}$ and $\varphi$ formulas, then  the mapping $||\cdot||:{\cal L}_{A}\to A$ is defined for closed formulas:

\begin{center}

$|| \bot||=0$,\\ [3mm]

$||u\in v || =\bigvee\limits_{x\in dom(v)} (v(x) \wedge ||x \approx u ||)  $ \\ [3mm]

$||u\approx v || =\bigwedge\limits_{x\in dom(u)} (u(x)) \to ||x\in v ||) \wedge \bigwedge\limits_{x\in dom(v)} (v(x) \to ||x\in u ||)$\\ [3mm]


$||\varphi \# \psi||  = ||\varphi||  \tilde{\#} ||\psi|| $, for every $\#\in \{\wedge,\vee, \to\}$,\\[3mm]



$||\exists x\varphi|| = \bigvee\limits_{{u\in \mathbf{V}^{\bf A}}} ||\varphi (u)||,$ \\[3mm]

 $||\forall x\varphi|| = \bigwedge\limits_{{u\in \mathbf{V}^{\bf A}}} ||\varphi (u)||$.

$||\varphi||$ is called the {\bf truth-value} of the sentence $\varphi$ in the language ${\cal L}_A$ in $\mathbf{V}^{\bf A}$ Boolean-valued model over ${\bf A}$.

\end{center}
\end{defi}

As usual, we abbreviate $\exists x(x \in u \wedge \Psi(x))$ by $\exists x \in u\, \Psi(x)$ and $\forall x(x\in u\, \to \Psi(x))$ by $\forall x \in u \Psi(x)$ and call these bounded quantifiers. We say that $\gamma$ is valid in $\mathbf{V}^{\bf A}$  if $||\gamma ||=1$ and write, $\mathbf{V}^{\bf A} \vDash \gamma$.
The basic system of Zermelo-Fraenkel set theory  here is called {\bf ZF} and consists of first order version  $QCL$ of Classical logic ($CL$) over the first-order signature $\Theta$ which contains an equality predicate $\approx$ and a binary predicate $\in $. The system {\bf ZF} is the first order theory with equality obtained from the logic $QCL$ over $\Theta$ by adding the following set-theoretic axiom schemas:
\begin{center}

(Extensionality) $\forall x \forall y [\forall z(z\in x \leftrightarrow z\in y)\to (z \approx y)]$ \\[3mm]

(Pairing)  $\forall x \forall y \exists w\forall z [z\in w \leftrightarrow (z \approx x \vee z \approx y) ]$ \\[3mm]

(Colletion) $\forall x [(\forall y\in x\exists z\phi(y,z))\to \exists w\forall y\in x \exists z\in w  \phi(y,z)]$\\[3mm]

(Powerset) $\forall x \exists w \forall z [z\in w \leftrightarrow \forall y\in z(y\in x)]$ \\[3mm]

(Separation) $\forall x \exists w \forall z[z\in w \leftrightarrow (z\in x \wedge \phi(z))]$ \\[3mm]

(Empty set) $\exists x\forall z[z\in x \leftrightarrow \neg (z\approx z)]$ 

The set satisfying this axiom is, by extensionality, unique and we refer to it with notation $\emptyset$.\\[3mm]

(Union) $\forall x \exists w \forall z[ z\in w \leftrightarrow  \exists y\in x(z\in y)]$ \\[3mm]

(Infinity) $\exists x [\emptyset\in x \wedge \forall y\in x (y^+\in x)]$

From union and pairing and extensionality, we can note by $y^+$ the unique set $y\cup \{y\}$. \\[3mm] 

(Induction) $\forall x [( \forall y\in x \phi(y))\to \phi(x)]\to \forall x\phi(x)$.

\end{center}

The original intuition of Boolean-valued models was that the names represent objects and that the equivalence classes of names under the equivalence relation defined by $u \sim  v$ if and only if $||u \approx v|| =1$ can serve as the ontology of the new model. In particular, this means that if two names represent the same object, they should instantiate the same properties. This is known as “indiscernibility of identicals”, one of the directions of Leibniz's Law. In our setting, we can represent this by a statement of the type 

$$||u \approx v|| \wedge ||\Psi(u)|| \leq ||\Psi(v)||.$$

Clearly, the  Boolean-valued models $\mathbf{V}^{\bf A}$  verify this Law.  So, we have the following

\begin{theo} All the axioms, hence all the theorems,  of {\bf ZF} are valid in $\mathbf{V}^{\bf A}$.  
\end{theo}

Now, replacing the Boolean algebra by a Heyting algebra, one obtains a Heyting-valued model. The proofs of the Boolean case transfer to the Heyting-valued, where the logic of the Heyting algebra determines the logic of the Heyting-valued model of set theory. This idea was further generalized by Takeuti \& Titani (1992); Titani (1999); Titani \& Kozawa (2003); and Ozawa (2017), replacing the Heyting algebra by appropriate lattices that allow models of quantum set theory or fuzzy set theory. After this, L\"owe and Tarafder proposed a class of reasonable implication algebra in order to construct algebraic-valued models that validate all axioms of the negation-free fragment of Zermelo-Fraenkel set theory, \cite{O-17, TK-03, T-99, TT-92}. From now on, we shall call this Set Theories as non-classical set Theories.

It is important to note that there are different kinds of models for the above non-classical Set Theories, see for instance, \cite{Fou-80, Fit-69, Bell-88}.

Zermelo-Fraenkel type set theories with models are the based on Intuitionistic, Fuzzy, Quantum logics and the family of set theories based on intermediate logics between the classical logics and the mentioned logics. Such non-classical set theories are particularly based on algebraizable logics and the models can be constructed over algebras that are semantics of such logics. 

Other models, such as the ones constructed over sheaf, topoi, possible world Kripke's semantics, or topological spaces are possible to be constructed due to the algebrizebility of the associated logics. While it is possible that these models need very sophisticated technical work, the existence of such models is strongly based on the fact that these set theories already have models of different nature.

Zermelo-Fraenkel's axioms are valid in  models of these set theories; that is to say, the set theories are sound from the logics  point of view. Even though we have not seen a proof the correctness, we think it is not possible to  give it due to a consequence of G\"odel’s incompletness theorem.

In general, the models of logic systems live in what philosophers call Meta-Mathematics and it is here where all mathematics mainly live; namely, analysis, algebra, topology and  whichever branch we know. This is exactly where G\"odel's proof lives. 

Although we can formalize new set theories, the models are construted over the set of recursive fucntions and the ordinal numbers of Meta-Mathematics. Thus, we think that G\"odel’s proof is possible to be given to each set theory; hence, these set theories are incomplete from this point of view.

Recall that G\"odel's proof of the famous incompletness theorem proved the indecibility of certain formulas of  Russell and Whitehead’s {\em Principia Mathematica}. Specifically, G\"odel  proves that there are properties of natural numbers that are true, but they can not be formally proved in the logical system. Russell and Whitehead's logical  system of Principia Mathematica was elaborated with the intention to make Hilbert's dream a reality. That is to say, to find a logical system where each mathematical theorem (of the Meta-Mathematics)  has a logical theorem that reflects it.

G\"odel’s proof shows the impossibility of that dream in the most simple mathematics, i.e. that of the natural numbers where certain recursive functions  are possible to be defined. Now is the right moment to ask ourselves: What did Hilbert look for? A model for all mathematics? The logic provides the models to other disciplines such as Philosophy, Computer Science, Economics, Physics, and so on. In this sense, Is it  possible to have a model of all? We do not think so. Mathematical models are not more than simplifications of the area  we want to model; indeed, very useful ones which allow a measure forward in knowledge, but the reality to be modeled is much more complex than the model can express and mathematics is not the exception. 

It is worth mentioning that in the non-classical Set Theories mentioned above, the negation formulas are positive formulas; indeed, $\neg x:= x\to \bot$.  Leibniz's law is verified for the positive formulas for each of these Set Theories. It is easy to see that all these systems can not produce {\em paraconsistency}; i.e.,   either only a formula, or its negation is valid.

Our reflexions over Meta-Mathematics are part of our research, allowing an undertanding and development of our objectives. Initially, we believed that to work in Paraconsistent Set Theory we would have to develop a kind of theory of recursive funtions and  transfinite recursion. As above  said, we have at our disposal these tools in  Meta-Mathematics.

\section{Paraconsistent Set Theory and  Leibniez' Law}\label{section3}

Going back to our topic, type ZF Paraconsistent Set Theories (PSTs) can be constructed in two groups; namely, one where Russell's paradox is accepted and another where the theories can be constructed using a paraconsistent logic.

Paraconsistent set theories of the first group have been studied by many authors  (Brady, 1971; Brady \& Routley, 1989; Restall, 1992; Libert, 2005; Weber, 2009, 2010, 2013); all of these accounts start from the observation that ZF was created to avoid the contradiction that can be obtained from the axiom scheme of Comprehension

$$\exists x\forall y(y \in x \leftrightarrow \Psi(y))$$

via Russell’s paradox. Arguing that contradictions are not necessarily devastating in a
paraconsistent setting, these authors reinstate the axiom scheme of Comprehension as acceptable, allow the formation of the Russell set R, and conclude that both $R \in R$ and $R \not\in R$ are true, see \cite{BR-89,Re-92,We-10a,We-10b,We-13,We-09}.

Now, on PST of the second group the following papers can be mentioned \cite{BT, CC-16}. In these works, the authors try to present models for certain PSTs type ZF (that we denote, PS3-ZF) in both papers models for positive fragment of PS3-ZF  are presented; i.e, formulas without paraconsistent negation. Besides, in these papers the philosophical concept called Leibniz Law (LL) is studied, observing that the LL is not verified for formulas with negation for PS3-ZF. Eve more, in \cite{GV}, the author proved that  the axiom  SEPARATION of ZF is not valid for formulas with negation. This fact shows that it is not possible to provide models in this way, but that does not mean that PS3-ZF has no models.

Moreover, we  have worked in other paraconsistent algebrizable logics without success, aiming to prove Leibniz law on a PST. Particularly, Priest's paraconsistent logics that he called da Costa logic were studied by us, see \cite{Pri-09,Pri-11}. This logic is algebrizable with the Blok-Pigozzi's method. We had realized that LL is an essential technical resource to  prove the soundness of all axioms of ZF. Later on, we decided to change the strategy using non-determinism to finding models, we had in our hands two kinds of  non-determinism, one from the Nmatrix of Avron and the other from Fidel's structures. Avron's  non-determinism is more unstable that of  Fidel's . The former is  a non-determinism via multialgebras. In the paper \cite{CFG-20}, we see that the formulas of first-order logic have a disconnect with his corresponding  interpretations,  producing a technical difficulty to give proofs. However, the latter, Fidel's non-determinism strongly uses the algebraic fragment of the system and formulas with negation have a valoration belonging to a certain algebra. These formulas with negation do not have an associated interpretation, but we can always assign a value of truth to them. Hence, as we can see in the paper \cite{FOS2},  Fidel's non-determinism is more stable. The associated interpretation of negation-free first-order formulas works exactly in the same way as the algebraic case. These formulas verify Leibniz law as the intuitionistic case. Constructing $F$-structures-valued models for Paraconsistent Set Theory type ZF based on da Costa’s Logic $C_\omega$ (ZF$C_\omega$), we can see that  Leibniz law is verified by formulas with negation that allow proving that all axioms of ZF are valid on Fidel's models . It is possible to assign  a value of truth belonging to a certain Heyting algebra to the formulas with negation. We do not known which is the real value, but we know there exists and it verifies the law.

On the other hand, we show that $C_\omega$ is not algebrizable in the Blok-Pigozzi’s sence; besides, we present a family of non-algebraic extentions of $C_\omega$ with Fidel's models for each of them. For each extention of $C_\omega$, we associate it a Paraconsistent Set Theory and present a full model for each one. It is worth mentioning that a very important philosophical concept as Leibniz law is definitely the only technical obstacle to getting full models; this fact will be part the our future studies. Another interesting aspect of these PSTs  do not permit Russell's paradox. Besides, the same strategy to use non-determinism can be applied for PS3-ZF. Actually, the logic PS3 seems to have two different F-strutures; namely, one based on Boolean algebras and another based on Heyting alegbras. Moreover, we can treat the algebrizable paraconsistent logics introduced by Priest in this way.

Now, we shall briefly present our results about ZF$C_\omega$ for the details the reader can consult our paper \cite{FOS2}. We fix a model of set theory V and a completed $C_\omega$-structure $(A, N)$. Let us construct a universe of names by transfinite recursion on $(A,N)$:

We fix a model of set theory $\mathbf{V}$ and a completed $C_{\omega}$-structure $\langle A,N\rangle$. Let us construct a universe of {\em names} by transfinite recursion: 
 
$$ {\mathbf{V}_\xi}^{\langle A, N\rangle}=\{x: x\, \textrm{\rm  a function and }\, ran(x)\subseteq A \, \,\textrm{\rm and }\,  dom(x)\subseteq \mathbf{V}_\zeta^{\langle A, N\rangle}\, \textrm{\rm for some}\,  \zeta< \xi \} $$

and 

$$ {\mathbf{V}}^{\langle A, N\rangle}=\{x: x\in {\mathbf{V}_\xi}^{\langle A, N\rangle}\, \textrm{\rm  for some } \xi \} $$

The ${\mathbf{V}}^{\langle A, N\rangle}$ is called the $C_{\omega}$-structure-valued model over $\langle A,N\rangle$. Let us observe that we only need set A in order to define $\mathbf{V}_\xi^{\langle A, N\rangle}$. By ${\cal L}_\in$, we denote the first-order language of set theory which consists of the propositional connectives $\{\to, \wedge, \vee, \neg\}$ of the $C_\omega$ and two binary predicates $\in$ and $\approx$. We can expand this language by adding all the elements of ${\mathbf{V}}^{\langle A, N\rangle}$; the expanded language we will denote ${\cal L}_{\langle A, N\rangle}$.

{\bf Induction principles:} The sets $\mathbf{V}_\zeta=\{x: x\subseteq \mathbf{V}_\xi, \, \textrm{ for some}\, \xi<\zeta\}$ are definable for every ordinal $\xi$ and then, every set $x$ belongs to $\mathbf{V}_\alpha$ for some $\alpha$. 

So, this fact induce a function $rank(x)=$ least ordinal $\xi$ such that $x\in \mathbf{V}_\xi$. Since $rank(x)<rank(y)$ is well-founded we induce a {\em principle of induction on rank}: let $\Psi$ be a property over sets. Assume, for every set $x$, if $\Psi(y)$ holds for every $y$ such that $rank(y)<rank(x)$, then $\Psi(x)$ holds. Thus, $\Psi(x)$ for every $x$. From the latter, the following (meta-)Induction Principles (IP) holds in ${\mathbf{V}}^{\langle A, N\rangle}$:  
\begin{center}\em 
Let us consider a property $\Psi$ over sets.  Assume, for every $x\in {\mathbf{V}}^{\langle A, N\rangle}$, if $\Psi(y)$ holds for every $y\in dom(x)$, then $\Psi(x)$ holds. Hence, $\Psi(x)$ holds for every $x\in {\mathbf{V}}^{\langle A, N\rangle}$.   
\end{center}

By simplicity, we note every set $u\in {\mathbf{V}}^{\langle A, N\rangle}$ by its name $u$ of ${\cal L}_{\langle A, N\rangle}$. Besides, we will write $\varphi(u)$ instead of $\varphi(x/u)$. Now, we are going to define a valuation by induction on the complexity of a closed formula in ${\cal L}_{\langle A, N\rangle}$.

\begin{defi}\label{str}

 For a given complete $C_\omega$-structure $\langle A, N\rangle$,  the mapping $||\cdot||:{\cal L}_{\langle A, N\rangle}\to \langle A, N\rangle $ is defined as follows:

\begin{center}

$||u\in v || =\bigvee\limits_{x\in dom(v)} (v(x) \wedge ||x \approx u ||)  $ \\ [3mm]

$||u\approx v || =\bigwedge\limits_{x\in dom(u)} (u(x)) \to ||x\in v ||) \wedge \bigwedge\limits_{x\in dom(v)} (v(x) \to ||x\in u ||)$\\ [3mm]


$||\varphi \# \psi||  = ||\varphi||  \tilde{\#} ||\psi|| $, for every $\#\in \{\wedge,\vee, \to\}$,\\[3mm]

$||\neg \alpha||\in N_{||\alpha||}$ and $||\neg\neg\alpha||\leq ||\alpha||$,\\ [2mm]


$||\exists x\varphi|| = \bigvee\limits_{{u\in {\mathbf{V}}^{\langle A, N\rangle}}} ||\varphi (u)||$ and $||\forall x\varphi|| = \bigwedge\limits_{{u\in {\mathbf{V}}^{\langle A, N\rangle}}} ||\varphi (u)||$.

$||\varphi||$ is called the {\bf truth-value} of the sentence $\varphi$ in the language ${\cal L}_{\langle A, N\rangle}$ in the $C_\omega$-structure-valued model over  $\langle A, N\rangle$.

\end{center}
\end{defi}

\begin{defi}
A sentence $\varphi$ in the language ${\cal L}_{\langle A, N\rangle}$  is said to be valid in ${\mathbf{V}}^{\langle A, N\rangle}$, which  is denoted by ${\mathbf{V}}^{\langle A, N\rangle} \vDash \varphi$, if $||\varphi||=1$.

\end{defi}

It is important to note that for every completed $C_\omega$-structure $\langle A, N\rangle$, the element $\bigwedge\limits_{x\in A} x$ is the first element of $A$ and so, $A$ is a complete Heyting algebra, we denote by ''$0$'' this element. Besides, for every closed formula $\phi$ of ${\cal L}_{\langle A, N\rangle}$ we have $||\phi||\in A$. Then, the following lemma has the same proof as intuitionistic set theory.

\begin{lemma}
For a given completed $C_\omega$-structure $\langle A, N\rangle$. Then,  $|| u \approx u ||=1$,  $u(x)\leq || x\in u||$ for every $x\in dom(u)$, and $||u=v||=||v=u||$, for every $u,v\in {\mathbf{V}}^{\langle A, N\rangle}$
\end{lemma}

\subsection{Ontological principle}

The {\em identity of indiscernibles} is an ontological principle that states that there cannot be separate objects or entities that have all their properties in common. To suppose that two things {\em indiscernible} is suppose they are the same thing under different names. 

A form of the principle {\em indiscernibility of identicals} is attributed to the German philosopher {\em Gottfried Wilhelm Leibniz}. In the non-classical  set theories, we have that the manes represent  objects and if we have  identical  objects they would have to have the same properties. This is known as {\em indiscernibility of identicals} and it could be considered as Leibniz's law by the following axiom:
 $$ u \approx v   \wedge \varphi(u) \to \varphi(v) $$

In the next, we are going to consider complete $C_\omega$-structures which verify the Leibniz's law. It is important to note that we have $C_\omega$-structures that verify this law,  it is enough to require $1\in N_x$ for all $x\in A$ for every $x\not=1$ and $0\in N_1$.

We will adopt the following notation, for every formula $\varphi(x)$ and  every $u\in \mathbf{V}^{\langle A, N\rangle}$:  $\exists x\in u \varphi(x)= \exists x(x\in u \wedge \varphi(x))$ and $\forall x\in u \varphi(x)= \forall x(x\in u \to \varphi(x))$. Thus, we have the following

\begin{lemma}\label{BQ}
Let $\langle A, N\rangle$ be a  complete Leibniz $C_\omega$-structure, for every formula $\varphi(x)$ and every $u\in \mathbf{V}^{\langle A, N\rangle}$ we have 
 $$|| \exists x\in u \varphi(x)||= \bigvee\limits_{x\in dom(u)} (u(x) \wedge || \varphi(x)||),$$ 
 $$|| \forall x\in u \varphi(x)||= \bigwedge\limits_{x\in dom(u)} (u(x) \to || \varphi(x)||).$$
\end{lemma}

The basic system of paraconsistent set theory here is called ZF$C_\omega$ and consists of first order version Q$C_\omega$ of $C_\omega$ over the first-order signature $\Theta_\omega$ which contains an equality predicate $\approx$ and a binary predicate $\in$.

\begin{defi} The system ZF$C_\omega$ is the first order theory with equality obtained from the logic Q$C_\omega$ over $\Theta_\omega$ by adding the following set-theoretic axiom schemas:
(Extensionality), (Pairing), (Colletion), (Powerset), (Separation), (Empty set), 
(Union), (Infinity) and (Induction). 
\end{defi}

\begin{theo} Let $\langle A, N\rangle$ be a complete $C_\omega$-structure such that ${\mathbf{V}}^{\langle A, N\rangle}$  satisfies Leibniz's Law. Then, the all set-theoretic axioms of ZF$C_\omega$ are valid in${\mathbf{V}}^{\langle A, N\rangle} \vDash \varphi$.
\end{theo}

\begin{coro} The axiom of scheme Comprehension is not valid in  ${\mathbf{V}}^{\langle A, N\rangle}$.
\end{coro}

It is enough to see that $||\exists x \forall y(y\in x)||=0$, this formula $\exists x \forall y(y\in x)$  is an instance of Comprehension.

\section{Leibniz's law and its models}\label{section4}

In this section we shal analyze the minimal conditions  for a models constructed over a Heyting algebras that we need to prove that several Zermelo-Fraenkel's set-theoretic axioms are valid in a suitable algebraic-like models.

We fix a model of set theory $\mathbf{V}$ and a completed reasonable implication algebra  $A$. Let us construct a universe of {\em names} by transfinite recursion: 
 
$$ {\mathbf{V}_\xi}^{A}=\{x: x\, \textrm{\rm  a function and }\, ran(x)\subseteq A \, \,\textrm{\rm and }\,  dom(x)\subseteq \mathbf{V}_\zeta^{A}\, \textrm{\rm for some}\,  \zeta< \xi \} $$

$$ {\mathbf{V}}^{A}=\{x: x\in {\mathbf{V}_\xi}^{A}\, \textrm{\rm  for some } \xi \} $$

The class ${\mathbf{V}}^{A}$ is called the algebraic-valued model over $A$. Let us observe that we only need the set $A$ in order to define  ${\mathbf{V}_\xi}^{\langle A, N\rangle}$.  By ${\cal L}_\in$, we denote the first-order language of set theory which consists of only the propositional connectives $\{\to, \wedge,\vee, \neg\}$ of the $C_\omega$ and two binary predicates $\in$ and $=$. We can expand this language by adding all  the elements of ${\mathbf{V}}^{\langle A, N\rangle}$; the expanded language we will denote ${\cal L}_{\langle A, N\rangle}$. For this construction  of models we also have {\bf Induction principles} as the case above.

Now, we shall consider a minimal requirement for defining value of truth of formulas in order to prove some of set-theoretic axiom of Zermelo-Freankel for Set Theory are valid. Now, for a given completed Heyting algebra  $A$,  the mapping $||\cdot||:{\cal L}_{A}\to A$ is defined as follow:

\begin{center}

$||u\in v || =\bigvee\limits_{x\in dom(v)} (v(x) \wedge ||x \approx u ||)  $ \\ [3mm]

$||u\approx u ||=1$

$||u\approx v || \leq ||\phi(u)\to \phi(v) ||$ for every formula $\phi$,\\ [3mm]

$ ||\neg\varphi ||= ||\varphi ||^\ast$,\\[3mm]

$||\varphi \# \psi||  = ||\varphi||  \tilde{\#} ||\psi|| $, for every $\#\in \{\wedge,\vee, \to\}$,\\[3mm]


$||\exists x\varphi|| = \bigvee\limits_{{u\in \mathbf{V}}^{A}} ||\varphi (u)||$ and $||\forall x\varphi|| = \bigwedge\limits_{{u\in \mathbf{V}}^{A}} ||\varphi (u)||$.

$||\varphi||$ is called the {\bf truth-value} of the sentence $\varphi$ in the language ${\cal L}_{ A}$ in the algebraic-valued model over  $A$.

\end{center}

\begin{defi}
A sentence $\varphi$ in the language ${\cal L}_{A}$  is said to be valid in ${\mathbf{V}}^{ A}$, which  is denoted by ${\mathbf{V}}^{A} \vDash \varphi$, if $||\varphi||=1$.

\end{defi}

\begin{lemma}\label{lema1}
For a given completed reasonable implication algebra $A$. Then,  $u,v\in {\mathbf{V}}^{A}$ we have

\begin{itemize}

\item[\rm (i)] $||u=v||=||v=u||$, 
\item[\rm (ii)] $u(x)\leq || x\in u||$ for every $x\in dom(u)$.

\end{itemize}
\end{lemma}

\begin{proof}
(i) Let us consider the formula $\phi(z):= u = z$, then $||u=v || \leq ||\phi(u)\to \phi(v) ||=||u=u||\to ||v=u||=1\to ||v=u||=  ||v=u||$. Analogously, we have $||v=u||\leq ||u= v ||$.

\

(ii) $|| x\in u||=\bigvee\limits_{z\in dom(u)} (u(z) \wedge ||z = x ||)\geq u(x) \wedge ||x = x ||=  u(x)$. \end{proof}

\

We will adopt the following notation, for every formula $\varphi(x)$ and  every $u\in \mathbf{V}^{\langle A, N\rangle}$:  $\exists x\in u \varphi(x)= \exists x(x\in u \wedge \varphi(x))$ and $\forall x\in u \varphi(x)= \forall x(x\in u \to \varphi(x))$. 

Now, we recall that for a given Heyting algebra $A$ we have the following properties hold: (P1) $x\wedge y\leq z$ implies $x\leq y \Rightarrow z$  and (P3) $y\leq z$ implies $z \Rightarrow x \leq y \Rightarrow x$ for any $x,y,z\in A$.

Thus, we have the following

\begin{lemma}\label{BQ}
Let $A$ be a  Heyting algebra algebra, for every formula $\varphi(x)$ and every $u\in \mathbf{V}^{A}$ we have 
 $$|| \exists x\in u \varphi(x)||= \bigvee\limits_{x\in dom(u)} (u(x) \wedge || \varphi(x)||),$$ 
 $$|| \forall x\in u \varphi(x)||= \bigwedge\limits_{x\in dom(u)} (u(x) \to || \varphi(x)||).$$
\end{lemma}
\begin{proof} Form the definition of $||\cdot||$ we have:

  $|| \exists x\in u \varphi(x)||= || \exists x (x\in u \wedge \varphi(x))||= \bigvee\limits_{v\in  \mathbf{V}^{A}} (||v\in u || \wedge || \varphi(v)||)=  \bigvee\limits_{v\in  \mathbf{V}^{A}} \bigvee\limits_{x\in dom(u)} (u(x)\wedge ||x=v|| \wedge || \varphi(v)||)=   \bigvee\limits_{x\in dom(u)} u(x)\wedge \bigvee\limits_{v\in  \mathbf{V}^{A}} (||x=v|| \wedge || \varphi(v)||)$.
  
  \
  
  Now, we have $|| v=x\wedge\varphi(v)||\leq ||\varphi(x)||$ and $|| x=x\wedge\varphi(x)||=||\varphi(x)||$. Therefore,
  
  $\bigvee\limits_{z\in dom(u)} u(x)\wedge \bigvee\limits_{v\in  \mathbf{V}^{A}} (||z=u|| \wedge || \varphi(v)||)= \bigvee\limits_{x\in dom(u)}  ||u(x)\wedge \varphi(x)||$.

\

On the other hand, 

$|| \forall x\in u \varphi(x)||= || \forall x(x\in u \to \varphi(x))||= \bigwedge\limits_{v\in  \mathbf{V}^{A}} ||v\in u|| \to ||\varphi(v)||$

Then, we have

$ \bigwedge\limits_{x\in dom(u) } [u(x)\to||\varphi(x) ||] \wedge ||v\in u||=\bigwedge\limits_{x\in dom(u) } [u(x)\to||\varphi(x) ||] \wedge \bigvee\limits_{x\in dom(u) } (u(x)\wedge  ||v=x||)= \bigvee\limits_{x\in dom(u) }(\bigwedge\limits_{x\in dom(u) } [u(x)\to||\varphi(x) ||] \wedge u(x)\wedge  ||v=x||)\leq(P5) \bigvee\limits_{x\in dom(u) } ||\varphi(x) ||\wedge  ||v=x||\leq ||\varphi(v) ||$

\

Form the latter and (P1), we can conclude that $ \bigwedge\limits_{x\in dom(u) } [u(x)\to||\varphi(x) ||] \leq ||v\in u||\to ||\varphi(v) ||$.

\

Now using Lemma 2.3 (ii) and (P3) we obtain 

$\bigwedge\limits_{v\in  \mathbf{V}^{A}} ||v\in u||\to ||\varphi(v) ||\leq \bigwedge\limits_{v\in  dom(u)} ||v\in u||\to ||\varphi(v) ||\leq \bigwedge\limits_{v\in  dom(u)} ||u(v)||\to ||\varphi(v) ||$. \end{proof}

\begin{defi}
Let $A$ be  a complete Heyting algebra. Given collection of sets $\{u_i: i\in I\} \subseteq \mathbf{V}^{ A}$ and $\{a_i: i\in I \}\subseteq A$, then mixture $\Sigma_{i\in I} a_i\cdot u_i$ is the fucntion $u$ with $dom(u)=\bigcup\limits_{i\in I} dom(u_i)$ and $u(x)= \bigvee\limits_{i\in I} a_i \wedge || x\in u_i||$.
\end{defi}

The following result is known as {\em Mixing Lemma} and its proof is exactly the same for intuitionistic case because it is an assertion about positive formulas. 

\begin{lem}
Let $u$ be the mixture $\Sigma_{i\in I} a_i\cdot u_i$. If $a_i\wedge a_j\leq || u_i = u_j||$ for all $i,j\in I$, then $a_i\leq || u_i = u||$.  
\end{lem}

 A set $B$ refines a set $A$ if for all $b\in B$ there is some $a\in A$ such that $b\leq a$. A Heyting algebra $H$ is refinable if every subset $A\subseteq H$ there exists some anti-chaim  $B$ in $H$ that refines $A$ and verifies $\bigvee A = \bigvee B$.

\begin{theo}
Let $A$ be a complete Heyting algebra such that $A$ is  refinable. If $\mathbf{V}^{A}\vDash \exists x \psi(x)$, then there is $u\in \mathbf{V}^{ A}$ such that $\mathbf{V}^{A}\vDash \psi(u)$. 
\end{theo}

Now, given a complete Heyting $ A'$ of $ A$, we have the associated models $\mathbf{V}^{\bf A'}$ and $\mathbf{V}^{\bf A}$. Then, it is easy to see that $\mathbf{V}^{\bf A'}\subseteq \mathbf{V}^{\bf A}$.

On the other hand, we say that a formula $\psi$ is {\em restricted} if all quantifiers are of the form $\exists y\in x$ or $\forall y\in x$, then we have  

\begin{lem}
For any complete Heyting algebra $ \bf A'$ of $ \bf A$ and any restricted negation-free formula $\psi(x_1,\cdots,x_n)$ with variables in  $\mathbf{V}^{ \bf A'}$ the equality $||\psi(x_1,\cdots,x_n)||^{ \bf A'}=||\psi(x_1,\cdots,x_n)||^{A}$.
\end{lem}

\

Next, we are going to consider the Boolean algebra ${\bf 2}=(\{0,1\},\wedge,\vee,\neg,0,1)$ and the natural mapping $\hat{\cdot }:  \mathbf{V}^{\bf A} \to \mathbf{V}^{ {\bf 2} }$ defined by $\hat{u}=\{\langle \hat{v}, 1\rangle: v\in u\}$. This is well defined by recursion on $v\in dom(u)$. Then, we have the following lemma holds:

\begin{lem}\label{tecnico}
\begin{itemize}
\item[\rm (i)] $|| u\in \hat{v}||=\bigvee\limits_{x\in v} || u=\hat{x} ||$ for all $v\in \mathbf{V}$ and $u\in \mathbf{V}^{\bf A }$,

\item[\rm (ii)] $u\in v \leftrightarrow \mathbf{V}^{\bf A } \vDash \hat{u}\in \hat{v}$ and $u=v \leftrightarrow \mathbf{V}^{\bf A } \vDash \hat{u}= \hat{v}$,

\item[\rm (iii)] for all $x\in \mathbf{V}^{ {\bf 2} }$ there exists a unique $v\in \mathbf{V}$ such that $\mathbf{V}^{ {\bf 2}}\vDash  x= \hat{v}$,

\item[\rm (iv)] for any formula negation-free formula $\psi(x_1,\cdots,x_n)$  and any $x_1, \cdots,x_n\in \mathbf{V}$, we have  $\psi(x_1,\cdots,x_n) \leftrightarrow \mathbf{V}^{ {\bf 2} } \vDash   \psi(\hat{x_1},\cdots,\hat{x_n})$. Moreover for any restricted negation-free formula $\phi$, we have $\phi(x_1,\cdots,x_n) \leftrightarrow \mathbf{V}^{\bf A} \vDash   \phi(\hat{x_1},\cdots,\hat{x_n})$. 

\end{itemize}

\end{lem}

The proof of the last theorem is the same for intuitionistic case because we consider restricted negation-free formulas and it will be used to prove the validity of axiom Infinity.

\subsection{ Validating axioms}

New, we are going to prove the validity of several set-theoretical axioms of  {\bf ZF} and let us consider a fix model $\mathbf{V}^{ A}$. Then:

 \

\noindent {\bf Pairing} 

Let $u,v\in \mathbf{V}^{\bf A}$ and consider the function $w=\{\langle u,1\rangle, \langle v,1\rangle \}$. Thus, we have that $|| z\in w||= (w(u)\wedge ||z=u||) \vee (w(v)\wedge ||z=v||)=  ||z=u||\vee ||z=v||= ||z=u\vee z=v||$.

\

\noindent {\bf Union} 

Given $u\in \mathbf{V}^{\bf A}$ and consider tha function $w$ with $dom(w)= \bigcup\limits_{v\in dom(u)} dom(v)$ and $w(x)= \bigvee\limits_{v\in A_x} v(x)$ where $A_x=\{v\in dom(u): x\in dom(v)\}$. Then,

\begin{eqnarray*}
 ||y\in w||  &= &\bigvee\limits_{x\in dom(w)} (||x=y|| \wedge \bigvee\limits_{v\in A_x} v(x))\\
&= & \bigvee\limits_{x\in dom(w)}  \bigvee\limits_{v\in A_x} (||x=y|| \wedge v(x) )  \\
&= & \bigvee\limits_{v\in dom(u)}  \bigvee\limits_{x\in dom(v)}(||x=y|| \wedge v(x) ) \\
&= & || \exists v\in u(y\in v)||.\\
\end{eqnarray*}

\

\noindent {\bf Separation}

Given $u\in \mathbf{V}^{\bf  A}$ and suppose $dom(w)=dom(u)$ and $w(x)=||x\in u||\wedge ||\phi(x)||$ then

\begin{eqnarray*}
  || z\in w|| &= & \bigvee\limits_{x\in dom(w)} (||y\in w||\wedge ||\phi(y)||\wedge ||y=z||)\\
&\leq & \bigvee\limits_{x\in dom(w)} (||\phi(z)||\wedge ||y=z||).  \\
\end{eqnarray*}

Besides, 

\begin{eqnarray*}
 ||\phi(z)||\wedge ||y=z|| &= &\bigvee\limits_{y\in dom(u)} (u(y)\wedge ||z=y|| \wedge ||\phi(z)||)\\
&\leq & \bigvee\limits_{y\in dom(u)} (||y\in u||\wedge ||z=y|| \wedge ||\phi(y)||)  \\
&= & \bigvee\limits_{y\in dom(u)} (w(y) \wedge ||z=y||) = ||z\in w||. \\
\end{eqnarray*}

\

\noindent {\bf Infinity} 

Assume the formula $\psi(x)$ is $\emptyset\in x \wedge \forall y\in x (y^+\in x)$. Then, the axiom in question is the sentence $\exists x \psi(x)$. Now, it is clear that the negation-free formula $\emptyset\in x \wedge \forall y\in x (y^+\in x)$ is restricted and certainly $\psi(\omega)$ is true. Hence, by Lemma \ref{tecnico} (iv), we get $||\psi(\hat{\omega})||=1$, and so, $||\exists x \psi(x)||=1$.

\

\noindent {\bf  Collection} 

Given $u\in \mathbf{V}^{\bf A }$ and $x\in dom(u)$ there exists by Axiom of Choice some ordinal $\alpha_x$ such that $\bigvee\limits_{y \in \mathbf{V}^{\bf A}} ||\phi(x,y)||= \bigvee\limits_{y \in \mathbf{V}^{\bf A}_{\alpha_x}} ||\phi(x,y)||$. For $\alpha=\{\alpha_x: x\in dom(u)\}$ and $v$ the function with domain  $\mathbf{V}^{\bf A }$ and range $\{1\}$, we have

\begin{eqnarray*}
  ||\forall x\in u\exists y \phi(x,y)  ||  &= & \bigwedge\limits_{x\in dom(u)} (u(x)\to \bigvee\limits_{y\in \mathbf{V}^{\langle A, N\rangle}} ||\phi(x,y)  ||) \\
&= & \bigwedge\limits_{x\in dom(u)} (u(x)\to \bigvee\limits_{y\in \mathbf{V}^{\langle A, N\rangle}_\alpha} ||\phi(x,y)  ||) \\
&=  & \bigwedge\limits_{x\in dom(u)}(u(x) \to ||\exists y\in v \phi(x,y)  ||) \\
 \\
 &= & ||\forall x\in u \exists y\in v \phi(x,y) ||    \\
 &\leq & ||\exists w \forall x\in u \exists y\in w \phi(x,y) ||. \\
\end{eqnarray*}

\section{First-order of the paracosinsistent Nelson's logic} 

Paracosinsistent Nelson's logic, for short PNL, was studied systematically by Odintsov. For more details and information of the issue the reader can consult Odintsov's book \cite{Od-08}. In the paper \cite{SA}, Akama considered at the first time the PNL in 1999. 

In this part of the paper, we shall present $F$-structures as semantics for first-order version of paracosinsistent Nelson's logic. First, assume the propositional signature propositional languages ${\cal L} = \{\vee, \wedge, \to, \neg, \bot\}$, where $\neg$ is a symbol for strong negation  as well as, the symbol $\forall$, universal quantifier, and $\exists$, existential quantifier,  together with punctuation marks, commas and parentheses. Besides, let $Var=\{v_1,v_2,...\}$ be a numerable set of individual variables. A first-order signature $\Theta$ is also composted by the pair $\langle\P,\F\rangle$, where $\P$ denotes a non-empty set of predicate symbols and $\F$ is a set of function symbols. The notions of bound and free variables, closed terms, sentences, and substitutability are  defined as usual. We denote by $\mathfrak{Fm}_\Theta$ over the set of all formulas of $\Theta$ and by $Ter$ the absolutely free algebra of terms. Sometimes, we say that $\mathfrak{Fm}_\Theta$ is the language over $\Theta$. 
By $\varphi(x/t)$ we denote the formula that results from $\varphi$ by replacing simultaneously all the free occurrences of the variable $x$ by the term $t$. The connectives of equivalence $\leftrightarrow$  and of strong equivalence
$\Leftrightarrow$ are defined as follows: $\psi \leftrightarrow \phi := (\psi \leftrightarrow \phi) \wedge (\phi \leftrightarrow \psi)$,  $\psi \Leftrightarrow \phi := (\psi \leftrightarrow \phi) \wedge (\neg \psi \leftrightarrow \neg\phi)$. As above, logics will be defined via Hilbert-style deductive systems with only the rules of substitution and modus ponens. In this way, to define a logic it is enough to give its axioms. First-order version of paraconsistent Nelson’s logic N4, for short QN4, is a logic in the language ${\cal L}$ characterized by the following list of axioms:

\noindent {\bf{Axioms}}

\begin{itemize}
  \item [(N1)] $\alpha\to(\beta\to \alpha)$,
  \item [(N2)] $(\alpha\to(\beta\to \gamma))\to((\alpha\to \beta)\to(\alpha\to \gamma))$,
  
  \item [(N3)] $(\alpha\wedge \beta)\to \beta$,  
  \item [(N4)] $(\alpha\wedge\beta)\to\alpha$,
  \item [(N5)] $\alpha\to(\beta\to(\alpha\wedge\beta))$,
  \item [(N6)] $\alpha\to(\alpha\vee\beta)$,
  \item [(N7)] $\beta\to(\alpha\vee\beta)$,
  \item [(N8)] $(\alpha\to\gamma)\to((\beta\to\gamma)\to((\alpha\vee\beta)\to\gamma))$,
  \item[(N9)] $\sim(\alpha\to \beta)\leftrightarrow \alpha\,\wedge\sim \beta$,
  \item[(N10)] $\sim(\alpha\wedge\beta)\leftrightarrow\,\sim\alpha\,\vee\sim\beta$,
  \item[(N11)] $\sim(\alpha\vee\beta)\leftrightarrow\,\sim\alpha\,\wedge\sim\beta$,
  \item[(N12)] $\sim(\neg\alpha)\leftrightarrow\,\alpha$,
  \item[(N13)] $\sim(\sim\alpha)\leftrightarrow\,\alpha$,  
  
   \item[(A1)] $\varphi(x/t)\to\exists x\varphi$, if $t$ is a term free for $x$ in $\varphi$,
  \item[(A2)] $\forall x\varphi\to\varphi(x/t)$, if $t$ is a term free for $x$ in $\varphi$,
\end{itemize}

\noindent{\bf Inference rules}

\begin{itemize}
  \item[(MP)] $\dfrac{\alpha, \alpha\to \beta}{\beta}$, 
   \item[(R3)] $\dfrac{\alpha\to\beta}{\exists x\alpha\to\beta}$, and $x$ does not occur free in $\beta$,
  \item[(R4)] $\dfrac{\alpha\to\beta}{\alpha\to\forall x\beta}$, and $x$ does not occur free in $\alpha$.
 
\end{itemize}

It is worth mentioning that in the propositional setting if we take the axioms from (N1) to (N13) with the rule {\em modus ponens}  we have the propositional logic N4. Besides, if we add the axiom (N14) $\sim\alpha\to(\alpha\to \beta)$  we have Nelson logics that  is known as N3, see \cite[Section 8.2]{Od-08}. Now, we introduce a class of $F$-structures that will serve as semantics for QN4. First, recall that Fidel presented for the first time $F$-structures as semantics for N3 in \cite{F2}.

Now, recall that an algebra   $\mathcal{A}=\langle A,\vee,\wedge,\to,0,1\rangle$ is said to be a {\em Heyting algebra} if the reduct $\langle A,\vee,\wedge,0,1\rangle$  is a bounded distributive lattice and the condition $x\wedge y\leq z$ iff $x \leq y\to z$ ($\ast$) holds.  Besides, the algebra $\mathcal{A}=\langle A,\vee,\wedge,\to,1\rangle$ is said to be  {\em generalized Heyting} algebra  if the reduct $\mathcal{A}=\langle A,\vee,\wedge\rangle$ it is a distributive lattice and  $\ast$ is verified.

\begin{defi}\label{defestruc} A $F$-structure for N4  is a  system $\langle A,\{N_x\}_{x\in A}\rangle$ where $A$ is a generalized Heyting algebra and $\{N_x\}_{x\in A}$ is a family of set of $A$ such that the following conditions hold:
\begin{itemize}
  \item[\rm (i)] for any $x\in A$, $N_x\not=\emptyset$,
  \item[\rm (ii)] for any $x,y\in A$, $x'\in N_x$ and $y'\in N_y$, the following relations hold $x'\vee y'\in N_{x\wedge y}$ and $x'\wedge y'\in N_{x \vee y}$, $x\in N_{x'}$,
  \item[\rm (iii)] for any $x,y\in A$, $y'\in N_x$, we have $x\wedge y' \in N_{x\to y}$.
\end{itemize}

\end{defi}

We are going to use the convention of algebraic logic,  we will write sometimes $\langle A,N\rangle$ instead of $\langle A,\{N_x\}_{x\in A}\rangle$. Besides, we call the  $F$-structures for N4 by N4-structures. As example of N4-structure, we can take a generalized Heyting algebra $A$ and the set $N_x^s=\{y\in A:x\vee y=1\}$. The structure $\langle A,\{N_x^s\}_{x\in A}\rangle$ will be said to be a saturated N4-structure.


The N4-structure $\langle A,\{N_x\}_{x\in A}\rangle$ is said to be a substructure of the N4-structure $\langle B,\{N'_x\}_{x\in B}\rangle$ if  $A$ is a subalgebra of  $B$ and $N_x\subseteq N'_x$ holds for $x\in A$. It is easy to see all N4-structure $\langle A,\{N_x\}\rangle$ is a substructure of $\langle A,\{N_x^s\}\rangle$ defined before.

\begin{defi} A $\Theta$-structure $\mathfrak{A}$ is a pair $(\langle A,\{N_x\}_{x\in A}\rangle,\S ) $ where $\langle A,\{N_x\}_{x\in A}\rangle$ is a completed N4-structure; i.e., $A$ is a completed generalized Heyting algebra. Besides, $\S=\langle S,\{P_{\S}\}_{P\in\P},\{f_{\S}\}_{f\in\F}\rangle$ is composted by  a non-empty domain $S$,  a function $P_{\S}:S^n\to \langle A,\{N_x\}_{x\in A}\rangle$, for each $n$-ary predicate symbol $P\in\P$, and  a function $f_{\S}:S^n\to S$, for each $n$-ary function symbol $f\in\F$.
\end{defi}

We are going to consider the usual notion of derivation of a formula $\alpha$ form $\Gamma$ in QN4 and we denote by  $\Gamma\vdash\alpha$. Besides, for a given $\Theta$-structure $\mathfrak{A}= (\langle A,\{N_x\}_{x\in A}\rangle ,\S)$, we say that a  mapping $v:Var\to S$ is   a $\mathfrak{A}$-valuation, or simply a valuation. By $v[x\to a]$ we denote the $\mathfrak{A}$-valuation where $v[x\to a](x)=a$ and $v[x\to a](y)=v(y)$ for any $y\in Var$ such that $y\neq x$.

\begin{defi} \label{estructura} Let $\mathfrak{A}=( \langle A,\{N_x\}_{x\in A}\rangle,\S)$ be a $\Theta$-structure and $v$ a $\mathfrak{A}$-valuation from $Var$ into $S$. We define the truth values of the terms and the formulas in $\mathfrak{A}$ for a valuation $v$ as follows:

\begin{center}

$||x||^\mathfrak{A}_v=v(x)$,\\ [2mm]

$||f(t_1,\cdots ,t_n)||^\mathfrak{A}_v=f_{\S}(||t_1||^\mathfrak{A}_v,\cdots ,||t_n||^\mathfrak{A}_v)$, for any $f\in\F$,\\ [2mm]

$||P(t_1,...,t_n)||^\mathfrak{A}_v=P_{\S}(||t_1||^\mathfrak{A}_v,...,||t_n||^\mathfrak{A}_v)$, for any $P\in\P$,\\ [2mm]

$||\varphi \# \psi||^\mathfrak{A}  = ||\varphi||^\mathfrak{A}  \# ||\psi||^\mathfrak{A} $, for every $\#\in \{\wedge,\vee, \to\}$,\\[2mm]

$||\neg \alpha||_v^\mathfrak{A}\in N_{||\alpha||_v^\mathfrak{A}}$ and $||\neg\neg\alpha||_v^\mathfrak{A}=||\alpha||_v^\mathfrak{A}$,\\ [2mm]

$||\neg (\alpha\vee \beta)||_v^\mathfrak{A}=  ||\neg \alpha||_v^\mathfrak{A} \wedge ||\neg \beta||_v^\mathfrak{A}$ and $||\neg(\alpha \wedge \beta)||_v^\mathfrak{A}=||\neg \alpha||_v^\mathfrak{A}\vee ||\neg \beta||_v^\mathfrak{A} $,\\ [2mm]

$||\neg(\alpha\to \beta)||_v^\mathfrak{A}=||\alpha||_v^\mathfrak{A}\wedge||\neg \beta||_v^\mathfrak{A} $,\\ [2mm]

$||\forall x\alpha||^\mathfrak{A}_v=\underset{a\in S}{\bigwedge} ||\alpha||^\mathfrak{A}_{v[x\to a]}$,\\ [2mm]

$||\exists x\alpha||^\mathfrak{A}_v=\underset{a\in S}{\bigvee}||\alpha||^\mathfrak{A}_{v[x\to a]}$, \\ [2mm]

$|| \alpha(x/t)||^\mathfrak{A}_v=||\alpha||^\mathfrak{A}_{v[x \to||t_1||^\mathfrak{A}_v]}$,  if $t$ is a term free for $x$ in $\varphi$. 

\end{center}
\end{defi}

\

It worth mentioning that the {\em substitution condition} $ ||\varphi (x/t) ||^\mathfrak{A}_v= ||\varphi ||^\mathfrak{A}_{v[x\to ||t||_v^\mathfrak{A} ]}$ can be proved for first order algebrizable logics. In our setting using $F$-structures for QN4 the negation-free formulas works exactly as the algebrizable case and the sustitution conditions holds, but for the atomic formulas with negation do not have an interpretation associated of them. Hence,  we need to impose the {\em substitution condition} as axiom as it was done for for da Costa's non-algebrizable paracosnistent logic $C_\omega$ in \cite{FOS2}.

\

Now, we say that $\mathfrak{A}$ and $v$ {\em satisfy} a formula $\varphi$, denoted by $\mathfrak{A}\vDash\varphi[v]$, if $||\varphi||^\mathfrak{A}_v=1$. Besides, we say that $\varphi$ is {\em true} $\mathfrak{A}$ if $||\varphi||^\mathfrak{A}_v=1$ for each  a $\mathfrak{A}$-valuation $v$ and we denote by $\mathfrak{A}\vDash\varphi$. We say that $\varphi$ is a {\em semantical consequence} of $\Gamma$ in QN4, if, for any structure $\mathfrak{A}$: if  $\mathfrak{A}\vDash\gamma$  for each $\gamma\in\Gamma$, then $\mathfrak{A}\vDash\varphi$. For a given set of formulas $\Gamma$, we say that the structure $\mathfrak{A}$  is a {\em model} of $\Gamma$ iff $\mathfrak{A}\vDash\gamma$  for each $\gamma\in\Gamma$.

Recall that  a logic  defined over a  language ${\cal S}$ is a system $\mathcal{L}=\langle For, \vdash\rangle$ where $For$ is the set of formulas over ${\cal S}$ and the relation  $\vdash \subseteq {\cal P} (For) \times For$ ( ${\cal P}(A)$ is the set of all subsets of $A$). The logic $\mathcal{L}$ is said to be a tarskian if it satisfies the following properties, for every set  $\Gamma\cup\Omega\cup\{\varphi,\beta\}$ of formulas:
\begin{itemize}	
  \item [\rm (1)] if $\alpha\in\Gamma$, then $\Gamma\vdash\alpha$,
  \item [\rm (2)] if $\Gamma\vdash\alpha$ and $\Gamma\subseteq\Omega$, then $\Omega\vdash\alpha$,
  \item [\rm (3)] if $\Omega\vdash\alpha$ and $\Gamma\vdash\beta$ for every $\beta\in\Omega$, then $\Gamma\vdash\alpha$.
\end{itemize}

\noindent A logic $\mathcal{L}$ is said to be finitary if it satisfies the following:

\begin{itemize}
  \item [\rm (4)] if $\Gamma\vdash\alpha$, then there exists a finite subset $\Gamma_0$ of $\Gamma$ such that $\Gamma_0\vdash\alpha$.
\end{itemize}

\begin{defi} \label{maxi} Let $\mathcal{L}$ be a tarskian logic and let $\Gamma\cup\{\varphi\}$ be a set of formulas, we say that $\Gamma$ is a theory. Besides,  $\Gamma$ is said to be a consistent theory if there is $\varphi$ such that $\Gamma\not\vdash_{\mathcal{L}}\varphi$. Besides, we say that $\Gamma$ is a maximal consistent theory if  $\Gamma,\psi\vdash_{\mathcal{L}}\varphi$ for any $\psi\notin\Gamma$ and in this case, we say $\Gamma$ non-trivial maximal respect to $\varphi$.
\end{defi}

A set of formulas $\Gamma$ is closed in $\mathcal{L}$ if the following property holds for every formula $\varphi$: $\Gamma\vdash_{\mathcal{L}}\varphi$ if and only if $\varphi\in\Gamma$. It is easy to see that any maximal consistent theory is closed one.

\begin{lem}[Lindenbaum-\L os] \label{exmaxnotr} Let $\mathcal{L}$ be a tarskian and finitary logic. Let $\Gamma\cup\{\varphi\}$ be a set of formulas  such that $\Gamma\not\vdash\varphi$. Then, there exists a set of formulas $\Omega$ such that $\Gamma\subseteq\Omega$ with $\Omega$ maximal non-trivial with respect to $\varphi$ in $\mathcal{L}$.
\end{lem}
\begin{proof}
It can be found \cite[Theorem 2.22]{W}.
\end{proof}

\

It is clear that QN4 is a finitary and tarskian logic. So, we are in conditions to show the following adequacy theorem. First, we can observe that for given a formula $\varphi$ and suppose $\{x_1,\cdots,x_n\}$ is the set of variable of $\varphi$, the {\em universal closure} of $\varphi$ is defined by $\forall x_1\cdots \forall x_n\varphi$. Thus, it is clear that if $\varphi$ is a sentence then the universal closure of $\varphi$ is itself.

\begin{theo}\label{compleprimord} Let $\Gamma\cup\{\varphi\}\subseteq\mathfrak{Fm}_\Theta$. Then, $\Gamma\vdash\varphi$  iff $\Gamma\vDash\varphi$.
\end{theo}
\begin{proof} 
We are going to consider a fixed structure  $\mathfrak{M}=\langle (A,N),\S\rangle$.  Let $\varphi$ be a formula such that $\Gamma\vdash\varphi$. Then, there exists $\alpha_1,\cdots ,\alpha_n$ a derivation of $\varphi$ from $\Gamma$. If $n=1$ then $\varphi$ is an axiom or $\varphi\in\Gamma$. If $\varphi\in\Gamma$, then it is easy to see that $\Gamma\vDash\varphi$. Besides, to prove the first-order version of each propositional axioms from N4 are valid is a routine, see for instance \cite{FOS1}. Now, for the sake of brevity we shall denote $||\varphi||_v$ instead of $||\varphi||_v^\mathfrak{M}$.

(A1) Suppose that $\varphi$ is $\alpha(t/x)\to\exists x\alpha$. Then, $||\varphi||_v=||\alpha||_{v[x\to||t||_v]}\to ||\exists x\alpha||_v$. It is clear that $||\alpha||_{v[x\to||t||_v]}\leq\underset{a\in S}{\bigvee}||\alpha||_{v[x\to a]}$ and then, $||\alpha||_{v[x\to||t||_v]}\leq||\exists x\alpha||_v$. Therefore $||\alpha(t/x)\to\exists x\alpha||_v=1$ and this holds for every valuation $v$. (A2) is analogous to (A1).

Suppose now that $||\alpha_j||_v=1$ for each $j<n$.

 If there exists $\{j,k_1,\cdots ,k_m\}\subseteq\{1,\cdots ,j-1\}$ such that $\alpha_{k_1},\cdots ,\alpha_{k_m}$ is a derivation of $\alpha\to\beta$. Let us suppose that $\varphi$ is $\exists x\alpha\to\beta$, where $x$ is not free in $\beta$, and it is obtained by applying $(\exists - In)$. From induction hypothesis $||\alpha\to\beta||_v=1$ for every valuation $v$. Now, consider  $||\exists x\alpha\to\beta||_v=||\exists x\alpha||_v\to||\beta||_v=\underset{a\in S}{\bigvee}||\alpha||_{v[x\to a]}\to||\beta||_v$. On the other hand, since $||\alpha\to\beta||_v=||\alpha||_v\to||\beta||_v=1$, then we have that $||\alpha||_v\leq||\beta||_v$ for each valuation $v$. Hence, $||\alpha||_{v[x\to a]}\leq||\beta||_{v[x\to a]} = ||\beta||_{v}$ for every $a\in S$ because $x$ is free in $\beta$. So, $\underset{a\in S}{\bigvee}||\alpha||_{v[x\to a]}\to||\beta||_{v}=||\exists x\alpha\to\beta||_v=||\varphi||_{v}=1$. The rest of the proof is left to the reader.

Conversely, let us suppose $\Gamma\vDash\varphi$ and $\Gamma\not\vdash\varphi$. Then, from the definition of $\vDash$, (A2) and $(\forall-In)$, we have $\forall\Gamma\vDash\forall\varphi$ and $\forall\Gamma\not\vdash\forall\varphi$ ($\ast$). From the latter and  Lindenbaum-\L os lemma, there exists $\Omega$ maximal  consistent theory  such that $\forall\Gamma\subseteq\Omega$ and $\Omega\not\vdash\varphi$. Let's consider the quotient algebra $A:=\mathfrak{Fm}_\Theta/\Omega$ where  $[\alpha]_\Gamma=\{ \beta \in \mathfrak{Fm}_\Theta: \Omega \vdash\alpha \to \beta, \Omega\vdash\beta \to \alpha\}$ is the class of $\alpha$ by $\Omega$. So, it is not hard to see $1 =[\beta]_\Gamma=\Gamma$ for every $\beta\in\Gamma$ (i.e. $\Gamma\vdash\beta$). It is clear that $A$ is a generalized Heyting algebra, and the a canonical projection $q:\mathfrak{Fm} \to A$ such that $q(\alpha)=[\alpha]_\Omega$  is a homomorphism such that $q^{-1}(\{1\})=\Omega$. Let us consider the $\Theta$-structure $\mathfrak{A}=( \langle A,\{N_x\}_{x\in A}\rangle, Ter)$ and let  $v:Var \to Ter$ be the identity function. So, we can consider $||.||_v: \mathfrak{Fm} \to \langle A, \{N_x\}_{x\in A}\rangle$ defined by $||\alpha||_v= [\alpha]_\Gamma$. Now, we have to  prove  $||\forall x\alpha||_v=\underset{a\in Ter}{\bigwedge} ||\alpha||_{v[x\to a]}$ and $||\exists x\alpha||_v=\underset{a\in Ter}{\bigvee}||\alpha||_{v[x\to a]}$. 
Indeed,  for any term $t$ we denote $\hat{t}$ the new constant. Now, from (A1) we have $\vdash \psi (x/\hat{t}) \to \exists x\psi$ for every term $t$ free for $x$ in $\psi$.  So, $\Gamma \vdash  \psi (x/\hat{t}) \to \exists x\psi$. Thus, $|| \psi (x/\hat{t}) \to \exists x\psi||_v= || \psi (x/\hat{t})||_v \to ||\exists x\psi||_v=|| \psi ||_{v[x\to ||\hat{t}||_v]} \to ||\exists x\psi||_v = || \psi ||_{v[x\to t]} \to ||\exists x\psi||_v= 1$ for every $t\in Ter$. Now, let us suppose there is sentence $\phi$ such that $|| \psi (x/\hat{t})||_v  \leq ||\phi||_v$ for every term in the some before condition; that is to say, $||\phi||_v$ is a upper bound of the set $\{|| \psi (x/\hat{t})||_v\}$ and $x$ is free in $\phi$. Thus,  $|| \psi (x/\hat{t}) \to \phi||_v=1$. and therefore, $\Gamma \vdash \psi (x/\hat{t}) \to \phi$ for every $t$ in the same condition.

In particular for $\hat{x}$, we have $||\alpha(x)||_{v}\leq   ||\beta ||_v$ where $v(x)= \hat{x}$. Therefore, $\Gamma \vdash \psi (x) \to \phi$. So, from (R3), we can infer that $\Gamma \vdash \exists x \psi (x)\to \phi$ and then, $|| \exists x \psi ||_v \leq ||\phi||_v$. Therefore, $||\exists x\alpha||_v=\underset{a\in Ter}{\bigvee}||\alpha||_{v[x\to a]}$. The rest of proof is completely analogous, but now by using (A2) and (R4). Therefore, $||\cdot||_v$ is a valuation  such that $[\psi]_\Omega=1$ iff $\Omega\vdash \psi$. Now, consider the complete lattice $A^\ast$ by MacNeille completion of $A$, see \cite{Birk}. Thus, consider the $\Theta$-structure $\mathfrak{A}^\ast=( \langle A^\ast,\{N_x\}_{x\in A^\ast}\rangle, Ter)$.  Now, since $\forall\Gamma$ is a set of sentences then $||\gamma||_v= ||\gamma||_\mu$ for every valuation $\mu$ and each $\gamma\in \forall\Gamma$. Therefore, by definition $\vDash$, we obtain that $\mathfrak{A}^\ast\vDash\gamma$  for each $\gamma\in\forall\Gamma$ but $\mathfrak{A}^\ast\not\vDash\forall\varphi$ which contradicts the statement ($\ast$).
 
 \end{proof}
\subsection{Paraconsistent Nelson's Set Theory}

The basic system of paraconsistent set theory here is called ZF-N4 and consists of first order  version QN4 of N4 over the first-order signature ${\Theta}_\omega$ which contains an equality predicate\, $\approx$ \, and a binary predicate $\in$. The system ZF-N4 is the first order theory with equality obtained from the logic  QN4 over $\Theta_\omega$  by adding the following set-theoretic axiom schemas: (Extensionality), (Pairing), (Colletion), (Powerset), (Separation), (Empty set), (Union), (Infinity) and (Induction), see Section \ref{section2}. 

Now, we construct the class ${\mathbf{V}}^{\langle A, N\rangle}$  of  N4-structure-valued model over $\langle A,N\rangle$ following Section \ref{section3}. By ${\cal L}_\in$, we denote the first-order language of set theory which consists of the propositional connectives $\{\to, \wedge, \vee, \neg\}$ of the N4 and two binary predicates $\in$ and $\approx$. We can expand this language by adding all the elements of ${\mathbf{V}}^{\langle A, N\rangle}$; the expanded language we will denote ${\cal L}_{\langle A, N\rangle}$. Now, we are going to define a valuation by induction on the complexity of a closed formula in ${\cal L}_{\langle A, N\rangle}$. Then, for a given complete N4-structure $\langle A, N\rangle$,  the mapping $||\cdot||:{\cal L}_{\langle A, N\rangle}\to \langle A, N\rangle $ is defined as follows:

\begin{center}

$||u\in v || =\bigvee\limits_{x\in dom(v)} (v(x) \wedge ||x \approx u ||)  $ \\ [3mm]

$||u\approx v || =\bigwedge\limits_{x\in dom(u)} (u(x)) \to ||x\in v ||) \wedge \bigwedge\limits_{x\in dom(v)} (v(x) \to ||x\in u ||)$\\ [3mm]

$||\varphi \# \psi||  = ||\varphi||  \tilde{\#} ||\psi|| $, for every $\#\in \{\wedge,\vee, \to\}$,\\[3mm]

$||\neg \varphi||_v^\mathfrak{A}\in N_{||\varphi||_v^\mathfrak{A}}$ and $||\neg\neg\varphi||_v^\mathfrak{A}=||\alpha||_v^\mathfrak{A}$,\\ [2mm]

$||\neg (\varphi\vee \psi)||_v^\mathfrak{A}=  ||\neg \varphi||_v^\mathfrak{A} \wedge ||\neg \psi||_v^\mathfrak{A}$ and $||\neg(\varphi\wedge \psi)||_v^\mathfrak{A}=||\neg \varphi||_v^\mathfrak{A}\vee ||\neg \psi||_v^\mathfrak{A} $,\\ [2mm]

$||\neg(\varphi\to \psi)||_v^\mathfrak{A}=||\varphi||_v^\mathfrak{A}\wedge||\neg \psi||_v^\mathfrak{A} $,\\ [2mm] 

$||\exists x\varphi|| = \bigvee\limits_{{u\in {\mathbf{V}}^{\langle A, N\rangle}}} ||\varphi (u)||$ and $||\forall x\varphi|| = \bigwedge\limits_{{u\in {\mathbf{V}}^{\langle A, N\rangle}}} ||\varphi (u)||$. \\ [2mm] 

$||u\approx v || \leq ||\neg\phi(u)||\to ||\neg\phi(v) ||$ for any formula $\phi$\\[2mm]

$||\varphi||$ is called the {\bf truth-value} of the sentence $\varphi$ in the language ${\cal L}_{\langle A, N\rangle}$ in the $C_\omega$-structure-valued model over  $\langle A, N\rangle$.

\end{center}

Now, we say that a sentence $\varphi$ in the language ${\cal L}_{\langle A, N\rangle}$  is said to be valid in ${\mathbf{V}}^{\langle A, N\rangle}$, which  is denoted by ${\mathbf{V}}^{\langle A, N\rangle} \vDash \varphi$, if $||\varphi||=1$. 

For every completed N4-structure $\langle A, N\rangle$, the element $\bigwedge\limits_{x\in A} x$ is the first element of $A$ and so, $A$ is a complete Heyting algebra, we denote by ''$0$'' this element. Besides, for every closed formula $\phi$ of ${\cal L}_{\langle A, N\rangle}$ we have $||\phi||\in A$ and so the proof of the following Lemma can be given ins the exactly same way that was done in Lemmas \ref{lema1}
\begin{lemma}
For a given completed N4-structure $\langle A, N\rangle$. Then,  $|| u \approx u ||=1$,  $u(x)\leq || x\in u||$ for every $x\in dom(u)$, and $||u=v||=||v=u||$, for every $u,v\in {\mathbf{V}}^{\langle A, N\rangle}$
\end{lemma}

The following fact can be checked by induction on the structure of formulas.

\begin{lemma}\label{LL} For any complete N4-structure the following Leiniz law: $||u\approx v || \leq ||\phi(u)\to \phi(v) ||$ for any formula $\phi$.
\end{lemma}

From the Lemmas \ref{BQ} and \ref{LL}, we have proven the following central result:

\begin{lemma}
Let $\langle A, N\rangle$ be a  complete Leibniz N4-structure, for every formula $\varphi(x)$ and every $u\in \mathbf{V}^{\langle A, N\rangle}$ we have 
 $$|| \exists x\in u \varphi(x)||= \bigvee\limits_{x\in dom(u)} (u(x) \wedge || \varphi(x)||),$$ 
 $$|| \forall x\in u \varphi(x)||= \bigwedge\limits_{x\in dom(u)} (u(x) \to || \varphi(x)||).$$
\end{lemma}

Taking into account the content of section \ref{section4}, we have proven the following Theorem. 

\begin{theo} Let $\langle A, N\rangle$ be a complete N4-structure. Then, the  set-theoretic axioms  (Pairing), (Colletion), (Separation), (Empty set), (Union), (Infinity) and (Induction) are valid in ${\mathbf{V}}^{\langle A, N\rangle} \vDash \varphi$.
\end{theo}

Now, we are in condition of proving the axioms (Extensionality) and (Powerset). Indeed,

\begin{theo} Let $\langle A, N\rangle$ be a complete N4-structure. Then, the  set-theoretic axioms  (Extensionality), (Powerset) and (Empty set) are valid in ${\mathbf{V}}^{\langle A, N\rangle} \vDash \varphi$.
\end{theo}
\begin{proof} Given $x,y\in \mathbf{V}^{\langle A, N\rangle}$, then

\begin{eqnarray*}
  ||\forall z (z\in x \leftrightarrow z\in y)|| &= & ||\forall z ((z\in x \to z\in y) \wedge (z\in y \to z\in x) ||\\
&= & \bigwedge\limits_{z\in \mathbf{V}^{\langle A, N\rangle}} (||z\in x|| \to ||z\in y||)\wedge \bigwedge\limits_{z\in \mathbf{V}^{\langle A, N\rangle}} (||z\in y|| \to ||z\in x||) \\
&\leq  & \bigwedge\limits_{z\in dom(x)} (||z\in x|| \to ||z\in y||)\wedge \bigwedge\limits_{z\in dom(y)} (||z\in y|| \to ||z\in x||) \\
&\leq & \bigwedge\limits_{z\in dom(x)} ( x(z) \to ||z\in y||)\wedge \bigwedge\limits_{z\in dom(y)} ( y(z) \to ||z\in x||)  \\
 &= & ||x=y|| \\
\end{eqnarray*}

Assume $u\in \mathbf{V}^{\langle A, N\rangle}$ and suppose $w$ a function such that $dom(w)=\{f:dom(u)\to A: f\, \hbox{function}\}$ and $w(x)=||\forall y\in x(y\in u)||$. Therefore, 

$$|| v\in w||=\bigvee\limits_{x\in dom(w)} (||\forall y\in x(y\in u)|| \wedge || x=v||)\leq ||\forall y\in v(y\in u)|| .$$. Thus, axiom  Extensionality is valid.

On the other hand,  given $v\in \mathbf{V}^{\langle A, N\rangle}$ and consider the function $a$ such that $dom(a)=dom(u)$ and $a(z)=||z\in u|| \wedge ||z\in v ||$. So, it is clear that $a(z)\to ||z\in v ||=1$ for every $z\in dom(a)$, therefore

\begin{eqnarray*}
  ||\forall y\in v(y\in u)||  &= & \bigwedge\limits_{y\in dom(v)} (v(y)\to  || y\in u||)\\
&= &  \bigwedge\limits_{y\in dom(v)} (v(y)\to  (|| y\in u|| \wedge v(y))) \\
&\leq  & \bigwedge\limits_{y\in dom(v)} (v(y)\to  a(y)) \\
&\leq & \bigwedge\limits_{y\in dom(v)} ( v(y) \to ||y\in a||)\wedge \bigwedge\limits_{z\in dom(a)} ( a(z) \to ||z\in v||)  \\
 &= & ||v=a|| \\
\end{eqnarray*}
Since $a(y)\leq ||y\in u||$ for every $y\in dom(a)$ then we have $||\forall y\in a(y\in u)||=1$. Now by construction we have that $a\in dom(w)$ and so,  $ ||\forall y\in v(y\in u)||\leq ||\forall y\in a(y\in u)|| \wedge ||v=a||= w(a) \wedge ||v=a||\leq ||v\in w||$. Therefore, the axiom (Powerset) holds.

Now, we show that (Empty set) is valid. Indeed, first let us note that $||u=u||=1$ for all $u\in \mathbf{V}^{\bf A}$ and then, $||\neg(u=u)||\in N_1$. Therefore, let us consider a function $w\in \mathbf{V}^{\bf A }$ such that  $u\in dom(w)$ and $ran(w)\subseteq \{||\neg(u=u)||\}$, then it is clear that $||u\in w|| = \bigvee\limits_{x\in dom(w)} (w(x)\wedge ||u=x||)=||\neg(u=u)||$  which completes the proof.

\end{proof}

It is worth mentioning that for proving the (Extensionality) and (Powerset) axioms we only need the definition of valuation for atomic formulas formed with the predicates $\in$ and $\approx$. For non-classical Set Theories this expression of the valuations permits to prove the Leibniz law, but if one treat with a different negation; that is to say, a negation that is not a positive formula, this law is not valid, then it is almost impossible to have more different algebraic models for the law. What show that the non-determism is inherent for Paraconsistent Set Theories. On the other hand, is it interesting or practical to have a logical system that does not verify the law? We do not think so. What means to have {\em identical object} that they have no the same properties?  The answer is in the Meta-Matematics, where the models to live, and it is there where the indentical object have the same properties. This show us that to understanding what the logical systems can express we need to have ''right'' models. 

\subsection*{Acknowledgments}
The author acknowledges the support of a post-doctoral grant 2016/21928-0 from S\~ao Paulo Research Foundation (FAPESP), Brazil.

\end{document}